\documentclass[11pt]{amsart}
\theoremstyle{plain}
\newtheorem{thm}{Theorem}[section]
\newtheorem{theorem}[thm]{Theorem}

\newtheorem{lemma}[thm]{Lemma}
\newtheorem{corollary}[thm]{Corollary}
\newtheorem{proposition}[thm]{Proposition}
\theoremstyle{definition}

\newtheorem{notation}[thm]{Notation}

\newtheorem{definition}[thm]{Definition}

\newtheorem{example}[thm]{Example}

\newtheorem{question}[thm]{Question}

\numberwithin{equation}{section}

\newcommand{\btheta}{{\Theta}}
\newcommand{\bw}{{\bf w}}
\newcommand{\bi}{{\bf i}}
\newcommand{\p}{\partial}
\newcommand{\dd}{{\check{\rm d}}}

\newcommand{\sC}{{\mathcal C}}

\newcommand{\sF}{{\mathcal F}}

\newcommand{\sK}{{\mathcal K}}
\newcommand{\sL}{{\mathcal L}}

\newcommand{\sN}{{\mathcal N}}
\newcommand{\sO}{{\mathcal O}}

\newcommand{\sS}{{\mathcal S}}

\newcommand{\sW}{{\mathcal W}}

\newcommand{\C}{{\mathbb C}}

\newcommand{\BH}{{\mathbb H}}
\newcommand{\I}{{\mathbb I}}

\newcommand{\BP}{{\mathbb P}}

\newcommand{\W}{{\mathbb W}}

\def\Sym{\mathop{\rm Sym}\nolimits}

\def\Hom{\mathop{\rm Hom}\nolimits}

\title[Unbendable rational curves and contact structures]{Varieties of minimal rational tangents of unbendable rational curves subordinate to contact structures}

\author{Jun-Muk Hwang}

\begin{document}

\begin{abstract}
A nonsingular rational curve $C$ in a complex manifold $X$ whose normal bundle is isomorphic to $$\sO_{\BP^1}(1)^{\oplus p} \oplus  \sO_{\BP^1}^{\oplus q}$$ for some nonnegative integers $p$ and  $q$ is called an unbendable rational curve on $X$.
Associated with it is
the variety of minimal rational tangents (VMRT) at a point $x \in C,$ which is the germ of submanifolds  $\sC^C_x \subset \BP T_x X$
consisting of tangent directions of small deformations of $C$ fixing $x$.
Assuming that there exists a distribution $D \subset TX$ such that all small deformations of $C$ are tangent to $D$, one asks what kind of submanifolds of projective space can be realized as the VMRT $\sC^C_x \subset \BP D_x$.
When $D \subset TX$ is a contact distribution, a well-known  necessary condition is   that $\sC_x^C$ should be Legendrian with respect to the induced contact structure on $\BP D_x$.
We prove that this is also a sufficient condition: we construct a complex manifold $X$ with a contact structure $D \subset TX$ and an unbendable rational curve $C \subset X$ such that  all small deformations of $C$ are tangent to $D$ and the VMRT $\sC^C_x \subset \BP D_x$ at some point $x\in C$ is projectively isomorphic to an arbitrarily given Legendrian submanifold. Our construction uses the geometry of contact lines on the Heisenberg group and a technical ingredient is
the symplectic geometry of distributions the study of which has originated from geometric control theory.
\end{abstract}

\maketitle


\section{Introduction}

A nonsingular rational curve $C$ on a complex manifold $X$ of dimension $n$ is said to be unbendable if its normal bundle is isomorphic to
\begin{equation}\label{e.p} \sO_{\BP^1}(1)^{\oplus p} \oplus  \sO_{\BP^1}^{\oplus (n-1-p)}\end{equation} for some nonnegative integer $p$. From the type of the normal bundle, deformations of $C$ in $X$ are unobstructed and all small deformations of $C$ are unbendable rational curves. In terms of the Douady space ${\rm Douady}(X)$ of $X$, we can find an open neighborhood $\sK^C$ (or the germ ) of the point $[C]\in {\rm Douady}(X)$ corresponding to $C \subset X$ such that $\sK^C$ is smooth of dimension $$ \dim H^0(\BP^1, \sO_{\BP^1}(1)^{\oplus p} \oplus  \sO_{\BP^1}^{\oplus (n-1-p)}) = n-1+p$$ and any curve in $X$ corresponding to a point of $\sK^C$ is an unbendable rational curve. If furthermore, there exists a distribution $D \subset TX$ on $X$ such that the curve in $X$ corresponding to any point of $\sK^C$ is tangent to $D$,  we say that $C$ is subordinate to the distribution $D \subset TX$.

Unbendable rational curves  arise naturally in the study of  uniruled projective manifolds  in algebraic geometry, where they were sometimes called standard rational curves (e.g. Section 4.1 in \cite{HM98} or Definition 4.16 in \cite{Hw12}) and there are many examples subordinate to nontrivial distributions (e.g. Examples 1.4.6 and 1.4.7 in \cite{Hw01}).

An important invariant of an unbendable rational curve $C \subset X$ is its variety of minimal rational tangents (VMRT) at a point $x \in C$. This was originally defined for minimal rational curves in algebraic geometry, but the same idea works in the complex-analytic setting if we consider germs of submanifolds in the projectivized tangent space as follows.

\begin{definition}\label{d.vmrt}
Let $C \subset X$ be an unbendable rational curve  and let $\sK^C$ be an open neighborhood of $[C]$ in the Douady space of $X$ considered above. After shrinking $\sK^C$ if necessary, we can assume that for any point $x \in C$, the subvariety $\sK^C_x \subset \sK^C$ of deformations of $C$ fixing the point $x$ is smooth of dimension $p$ and the subset $\sC_x^C \subset \BP T_x X$ consisting of the tangent directions to curves in $X$ belonging to $\sK_x^C$ is a locally closed complex submanifold of dimension $p$ in $\BP T_x X$. This submanifold $\sC^C_x$ is called the {\em VMRT} (standing for Varieties of Minimal Rational Tangents) of $C$ at $x$. By abuse of terminology, the VMRT at $x$ also refers to the   germ of $\sC_x^C$ at the point $\BP T_x C \in \sC_x^C,$ which is determined independently from the choice of $\sK^C$.
\end{definition}

  The projective equivalence class of the VMRT as a submanifold of the projective space is an invariant of $C \subset X$, in the sense that if there exists an unbendable rational curve $\tilde{C} \subset \tilde{X}$ with a biholomorphic map $\varphi: X \to \tilde{X}$ satisfying $\varphi(C) = \tilde{C}$ and $\varphi(x) = \tilde{x}$, then there exists a projective isomorphism $\BP T_x X \cong \BP T_{\tilde{x}}\tilde{X}$ sending $\sC^C_x$ to $\sC^{\tilde{C}}_{\tilde{x}}.$
For this reason, the projective geometric properties of the VMRT play important roles in the study of unbendable rational curves.

We are interested in the projective geometric nature of  locally closed submanifolds in projective space that arise as the VMRT's of  unbendable rational curves subordinate to distributions. If an unbendable rational curve $C \subset X$ is subordinate to a distribution $D \subset TX$, then there is an obvious inclusion $\sC^C_x \subset \BP D_x$. Is there any restriction on  a germ of submanifolds in $\BP^{r-1}$ that can be realized as the VMRT $\sC^C_x \subset \BP D_x$ at some point $x$ of an unbendable rational curve $C \subset X$ subordinate to a distribution $D \subset TX$ of rank $r$?
The following  restriction was discovered in \cite{HM98}.  Recall (see Definition \ref{d.Levi}) that for a distribution $D \subset TX$ on a complex manifold $X$ and a point $x \in X$, there is a naturally defined homomorphism
$${\rm Levi}^D_x: \wedge^2 D_x \to T_x X/D_x$$ induced by the Lie brackets of local sections of $D$ and called the Levi tensor of $D$ at $x$.
A projective line in $\BP D_x$ determines a 2-dimensional subspace in $D_x$, hence a 1-dimensional subspace in $\wedge^2 D_x$.  It was proved in Proposition 10 in \cite{HM98} that the 1-dimensional subspaces of $\wedge^2 D_x$ corresponding to projective lines in $\BP D_x$ tangent to the VMRT $\sC^C_x$ of an unbendable rational curve $C$ subordinate to $D$ are annihilated by ${\rm Levi}^D_x$. This is a  restriction on the projective geometry of the VMRT $\sC^C_x$, which is a nontrivial condition if  the Levi tensor is nonzero. One can find many interesting consequences of this condition in \cite{HM98} and \cite{Hw01}.

 Is there any other restriction? No additional restriction is  known, which motivates the following question,  Problem 6.3 in \cite{Hw12}.

\begin{question}\label{q.1}
Let $V$ be a vector space of dimension $r$ and let $S \subset \BP V$ be  a locally closed submanifold of dimension $p>0$  such that lines tangent to $S$ span a linear subspace $V^S \subset \wedge^2 V$ of codimension $c\geq 0$. Is there a triple $(X, D, C)$ consisting of a complex manifold $X$ of dimension $\geq r+c$, a distribution $D \subset TX$ of rank $r$ and an unbendable rational curve $C \subset X$ with $p$ as in (\ref{e.p}) such that
\begin{itemize} \item[(1)]  $C$ is subordinate to $D$;
\item[(2)] there exists a projective isomorphism $\varphi:  \BP D_x \to \BP V$ for some point $x \in C$ satisfying $\varphi (\sC^C_x) \subset S$; and
\item[(3)] the isomorphism $\wedge^2 D_x \to \wedge^2 V$ induced by $\varphi$ in (2) sends the kernel of the Levi tensor $${\rm Levi}_x^D: \wedge^2 D_x \to T_x X/D_x$$ isomorphically to $V^S \subset \wedge^2 V$? \end{itemize} \end{question}

When $c=0$, the answer is easy.  In fact, regarding $\BP V$ as a hyperplane in $\BP^r$, let $X$ be the blowup of $\BP^r$ along $S$. Then the proper image $C \subset X$ of a projective line $\BP^1 \subset \BP^r$  intersecting both $S \subset \BP V \subset \BP^r$ and $\BP^r \setminus \BP V$ becomes an unbendable rational curve on $X$ (see Example 5.6 in \cite{Hw12}) whose VMRT at a point over $\BP^r \setminus \BP V$ is isomorphic to the germ of $S$ at $\BP^1 \cap S $.  Then $(X, D, C)$ satisfies (2) in Question \ref{q.1}, while (1) and (3) are automatic in this case.

There has been little progress on Question \ref{q.1} when $c>0$.  The above mentioned construction of the blowup of $\BP^r$ when $c=0$ can be regarded as a partial completion of the vector group $V$ acting transitively on $\BP^r \setminus \BP V.$ So we expect that an answer to Question \ref{q.1} may be found as a partial completion of the nilpotent group associated with the 2-step graded Lie algebra $V \oplus \wedge^2 V / V^S$. But it is not clear how to construct such a partial completion  generalizing the blowup construction of the case  $c = 0$.

In this paper, we work on the simplest unknown case of $c=1$ and $\dim X = r+1$ in Question \ref{q.1}, assuming furthermore that the hyperplane  $V^S \subset\wedge^2 V$  is given by a nondegenerate antisymmetric form $\omega$ on $V$. In this case, the dimension of $V$ is an even integer, to be denoted by $2m>0$.  Then the affine cone $\hat{S} \subset V$ of $S$ is an isotropic subvariety of the symplectic vector space $(V, \omega)$.
If a triple $(X, D, C)$ satisfies the three conditions of Question \ref{q.1}, then $D$ must be a contact structure on $X$. In terms of the line bundle $L := TX /D$,  the anti-canonical bundle $K^{-1}_X = \det TX$ must be isomorphic to $L^{\otimes (m+1)}$ (e.g. by (2.2) in \cite{LB}).
It follows that $C \cdot L =1$ and $\dim \hat{S} = p+1 = m$. This means that $\hat{S}$ must be a Lagrangian subvariety of the symplectic vector space $(V, \omega)$, or equivalently, the submanifold $S$ is a Legendrian submanifold of $\BP V$ equipped with the natural contact structure on $\BP V$ induced by $\omega$ (see Example \ref{e.contact} and Definition \ref{d.Legendrian}).
 We give an affirmative answer to Question \ref{q.1} in this case as follows.

\begin{theorem}\label{t.main}
Let $(V, \omega)$ be a symplectic vector space and let $S \subset \BP V$ be a locally closed submanifold which is Legendrian with respect to the contact structure on $\BP V$ induced by $\omega$. Then we have a complex manifold $X$ equipped with a contact structure $D \subset TX$ and an unbendable rational curve $C \subset X$ subordinate to $D$ such that for some point $x \in C$, there exists a projective isomorphism $\varphi:\BP D_x \to \BP V$ sending $\sC^C_x$ to an open subset of  $S.$
\end{theorem}

In addition to Question \ref{q.1}, Theorem \ref{t.main} is motivated by the well-known problem whether any Fano contact manifold should be the twistor space of a quaternionic-K\"ahler manifold (\cite{LB}).
To approach  this problem via the theory of VMRT, we expect that a neighborhood of $C \subset X$ constructed in Theorem \ref{t.main} should serve as the model geometry.

When $S \subset \BP V$ is a certain quasi-homogeneous Legendrian submanifold determined by a cubic form, a construction
of $C \subset X$ and $D \subset TX$ as in Theorem \ref{t.main} was given in \cite{HwMa}. The method of \cite{HwMa} cannot be extended to a general Legendrian submanifold $S \subset \BP V$ and we use a  completely different argument here.  Our $X$ in Theorem \ref{t.main} is  a partial completion of the Heisenberg group along a chosen contact line (see Theorem \ref{t.precise} for a precise statement). The VMRT $\sC^C_x \subset \BP D_x$ is isomorphic to the germ of $S \subset \BP V$ at the direction of the chosen contact line for all $x \in C$,  except possibly for a single special point in $C$.  We employ the contact geometry of a natural distribution on the space of contact lines on the Heisenberg group (see Section \ref{s.Heis} for details). This method is inspired by the works \cite{AS} and \cite{Z99} in control theory on the symplectic geometry of distributions, part of which we review in Section \ref{s.dist}.
Our construction is explained in Section \ref{s.Heis} modulo some technical computations. The latter is carried out in Section \ref{s.coordi} by employing a suitable coordinate system on the Heisenberg group.

\section{Preliminaries on symplectic geometry of cotangent bundles}\label{s.cotan}
We recall some basic definitions and results on the natural symplectic form on the cotangent bundle of a complex manifold.

\begin{definition}\label{d.form}
Let $L$ be a line bundle on a complex manifold $M$.

 \begin{itemize} \item[(1)] A homomorphism $\vartheta: TM \to L$ is called an $L$-{\em valued form}. We say that $\vartheta$ is {\em nowhere vanishing} if the homomorphism $TM \to L$ is surjective. \item[(2)] For a nowhere vanishing $L$-valued form $\vartheta$ on $M$, define a homomorphism $$\dd_x \vartheta: \wedge^2 {\rm Ker}(\vartheta_x) \to L_x$$ of the fibers of vector bundles at $x \in M$ as follows. Choose a section $l$ of $L$  in a neighborhood $O$ of $x$ with $l_x  \neq 0$ such that $\vartheta|_O =\check{\vartheta} \otimes l $ for some 1-form $\check{\vartheta}$ on $O$. Define
 $$ \dd_x \vartheta (u, v) := {\rm d} \check{\vartheta} (u, v) \cdot l_x$$ for $u, v \in {\rm Ker}(\vartheta_x)\subset T_x M.$ It is easy to see that this definition does not depend on the choice of $l.$
 \item[(3)] In (2), define ${\rm Null}(\dd_x \vartheta) \subset {\rm Ker}(\vartheta_x)$ by
     $${\rm Null}(\dd_x \vartheta) := \{ v \in {\rm Ker}(\vartheta_x), \dd_x \vartheta (u, v) = 0 \mbox{ for all } u \in {\rm Ker}(\vartheta_x).\}.$$
 \item[(4)] In (3), we say that $\vartheta$ is a {\em contact form} on $M$ if ${\rm Null}(\dd_x \vartheta) =0$ for all $x \in M$. The distribution ${\rm Ker}(\vartheta) \subset TM$ of corank 1 determined by a contact form is called a {\em contact structure} on $M$. \end{itemize} \end{definition}

\begin{example}\label{e.contact}
Let $V$ be a vector space with a symplectic form $\omega: \wedge^2 V \to \C$.
The projectivization $\BP V$ is the set of $1$-dimensional subspaces in $V$.
Denote by $[v] \in \BP V$ the point corresponding to a nonzero vector $v \in V$.
There is a natural identification of the tangent space $$T_{[v]} \BP V  \ = \ \Hom (\C v, V/\C v).$$
Let $L$ be the line bundle $\sO_{\BP V}(2)$ on $\BP V$ such that
 the fiber of $L$ at $[v] \in \BP V$ with $v \in V\setminus \{0\}$ is $\Sym^2 (\C v)^{\vee}$. Define $\vartheta: T \BP V \to L$ by
sending $h \in T_{[v]}\BP V = \Hom (\C v, V/\C v)$ to $\vartheta(h) \in \Sym^2(\C v)^{\vee}$ defined by $$ \vartheta(h) (v) = \omega(h(v), v).$$ This is a contact form on $\BP V$ (as in Example 2.1 of \cite{LB}) defining a contact structure on $\BP V$ naturally induced by $\omega.$ \end{example}

The following lemma is straightforward.

\begin{lemma}\label{l.pullback}
Let $\vartheta$ be a nowhere vanishing $L$-valued 1-form on a complex manifold $M$ as in Definition \ref{d.form}. Let  $\psi: Y \to M$ be a submersion from a complex manifold $Y$ and let $\psi^*\vartheta: TY \to \psi^*L$ be the pull-back of $\vartheta$. Then
$${\rm Null}(\dd_y (\psi^* \vartheta)) = ({\rm d}_y \psi)^{-1}({\rm Null}(\dd_x \vartheta))$$ for any $y \in \psi^{-1}(x) \subset Y$, where ${\rm d}_y \psi: T_y Y \to T_{\psi(y)} M$ is the differential of $\psi$ at $y$. \end{lemma}

\begin{notation}\label{n.cotangent}
Let $M$ be a complex manifold. \begin{itemize}
\item[(1)] For a vector  bundle $W$ on $M$, its projectivization $\BP W$ is the space of 1-dimensional subspaces in the fibers of $W \to M$. There is a natural line bundle $\sO_{\BP W}(1)$ on $\BP W$, whose fiber at the point $[w] \in \BP W$ corresponding to a nonzero vector $w \in W$ is $$\sO_{\BP W}(1)_{[w]} := \Hom (\C w, \C).$$
\item[(2)] Let $\pi_M: T^{\vee}M \to M$ be the cotangent bundle of $M$. We have a canonical 1-form $\theta^M$ on $T^{\vee}M$ which sends a tangent vector $v \in T_{\alpha}(T^{\vee}M)$ at $\alpha \in T^{\vee} M$ to
$$\theta^M(v) := \alpha ({\rm d} \pi_M (v)).$$
\item[(3)]
On the projectivization  $\varpi_M: \BP T^{\vee}M \to M$ of the cotangent bundle,  there exists a natural $\sO_{\BP T^{\vee}M}(1)$-valued 1-form $\vartheta^M$ which sends a tangent vector $u \in T_{[\alpha]}(\BP T^{\vee} M)$ with $0 \neq \alpha \in T^{\vee} M$ to $$\vartheta^M(u) := \alpha({\rm d} \varpi_M (u))
\cdot \alpha^{\vee}$$ where $\alpha^{\vee}$ is the element of $$\sO_{\BP T^{\vee}M}(1)_{[\alpha]} := \Hom (\C \alpha, \C)$$ satisfying $\alpha^{\vee}(\alpha) =1$. This definition does not depend on the choice of $\alpha$ representing the point $[\alpha]$. \end{itemize} It is well-known (e.g. Example 2.2 in \cite{LB}) that ${\rm d} \theta^M$ is a symplectic form on $T^{\vee}M$ and $\vartheta^M$ is a contact form on $\BP T^{\vee} M.$  \end{notation}

The following is from Proposition 11.14 of \cite{AS}.

\begin{proposition}\label{p.AS}
In Notation \ref{n.cotangent}, a vector field $\vec{v}$ on $M$ determines a function $h_{\vec{v}}$ on $T^{\vee} M$. Since ${\rm d} \theta^M$ is a symplectic form on $T^{\vee}M$, there is a unique vector field $\vec{h}_{\vec{v}}$ on $T^{\vee} M$ satisfying $${\rm d} \theta^M(\vec{h}_{\vec{v}}, \cdot) = - {\rm d} h_{\vec{v}}.$$
For a point $\alpha \in T_x^{\vee} M,$ let $\{\varphi_t, t \in \Delta\}$ be the 1-parameter group-germ of local biholomorphisms in a neighborhood of $x$ in $M$ generated by $\vec{v}$ for a suitable neighborhood $\Delta \subset \C$ of $0 \in \C$ and let $\{ \Phi_t, t \in \Delta\}$ be the 1-parameter group-germ of local biholomorphisms in a neighborhood of $\alpha$ in $T^{\vee} M$ generated by $\vec{h}_{\vec{v}}$. Then $$\varphi_t^* \beta = \Phi_{-t}(\beta)$$ for any $\beta \in T^{\vee}M$ close to $\alpha$ and $t \in \Delta$ close to $0$ such that both sides are well-defined. \end{proposition}

The following lemma is straightforward.

\begin{lemma}\label{l.sigma}
Let $M$ be a complex manifold and let $\theta$ be  a nowhere vanishing holomorphic 1-form on $M$. Define a section   $\sigma: M \to \BP T^{\vee} M$  of $\varpi_M$ in Notation \ref{n.cotangent} (3)  by $\sigma(x) = [\theta_x]$ for each $x \in M$.
Let $\tau$ be the  section of $\sO_{\BP T^{\vee}M}(1)|_{\sigma(M)}$ uniquely determined by the condition $$
\tau ([\theta_x]) \in \sO_{\BP T^{\vee}M}(1)_{[\theta_x]} = \Hom (\C \theta_x, \C)$$ sends $\theta_x$ to $1$.  Write  $\vartheta^M|_{\sigma(M)} = \check{\theta} \otimes \tau$ for some 1-form $\check{\theta}$ on   $\sigma(M).$ Then $\check{\theta}$ coincides with $\theta$ under the natural identification $\sigma(M) \cong M$ given by $\varpi_M$.  \end{lemma}

\section{Symplectic geometry of distributions}\label{s.dist}
We present some basic results on the symplectic geometry of distributions.

\begin{definition}\label{d.Levi}
Let $D \subset TM$ be a distribution on a complex manifold $M,$ meaning a  vector subbundle of the tangent bundle of  $M$.
\begin{itemize} \item[(i)] The Lie brackets of local sections of $D$ determine a homomorphism
$${\rm Levi}_x^D: \wedge^2 D_x \to T_x M/D_x,$$ called the {\em Levi tensor} (sometimes called O'Neill tensor or Frobenius tensor) of $D$ at $x$.  When $D$ is  ${\rm Ker}(\vartheta)$ for $\vartheta$ as in Definition \ref{d.form} (2), it is easy to see that ${\rm Levi}_x^D = - \dd_x \vartheta$ for each $x \in M$.
\item[(ii)] Denote by $D^{\perp} \subset T^{\vee}M$ the subbundle consisting of cotangent vectors annihilating elements of $D$.
\item[(iii)] For each $x \in M$ and  $\alpha \in D^{\perp}_x,$ use the natural identification $D^{\perp} = (TM/D)^{\vee}$ to define the anti-symmetric form  on $D_x$, $$  \alpha \circ {\rm Levi}_x^D: \wedge^2 D_x \to \C$$ and its null-space  by $${\rm Null}(\alpha \circ {\rm Levi}_x^D) = \{ v \in D_x, \ \alpha \circ {\rm Levi}_x^D (v, u) = 0 \mbox{ for all } u \in D_x\}.$$  \item[(iv)]
Let $\phi: Y \to M$ be a submersion from a complex  manifold $Y$. Denote by $T^{\phi} \subset TY$ the distribution on $Y$ given by ${\rm Ker} ({\rm d} \phi)$ and by $\phi^{-1}D \subset TY$  the distribution on $Y$ defined as the inverse image of $D$ under ${\rm d} \phi: TY \to TM$ such that its fiber at $y\in Y$ is given by
$$(\phi^{-1}D)_y = ({\rm d}_y \phi)^{-1}(D_{\phi (y)}).$$  It is clear that $T^{\phi} \subset \phi^{-1}D$. \end{itemize} \end{definition}

The following is straightforward.

\begin{lemma}\label{l.inverse}
In Definition \ref{d.Levi} (iv), denote by $$\phi_y: (TY/\phi^{-1}D)_y \to (TM/D)_x$$ the isomorphism induced by ${\rm d}_y \phi: T_yY \to T_x M$ for $y \in Y$ and $x= \phi(y)$. Then
$$\phi_y ( {\rm Levi}_y^{\phi^{-1}D}(u,v) ) = {\rm Levi}_x^D( {\rm d}_y \phi(u), {\rm d}_y \phi (v))$$
 for any $u,v \in (\phi^{-1}D)_y$.
\end{lemma}

\begin{lemma}\label{l.quotient}
Let $\phi: Y \to M$ be a submersion with connected fibers between complex manifolds.
Suppose that $R \subset TY$ is a distribution on $Y$ satisfying
$$T_y^{\phi} \subset R_y \mbox{ and } {\rm Levi}_y^R(u, v) =0 $$  for all $y \in Y,  u \in T_y^{\phi}$ and $ v \in R_y$. Then there exists a distribution $Q \subset TM$ such that $R = \phi^{-1}Q$. \end{lemma}

\begin{proof}
Write $r= {\rm rank}(R) - {\rm rank}(T^{\phi}).$
Let $\beta: {\rm Gr}(r, TM)\to M$ be the fiber bundle with the fiber $\beta^{-1}(x)$ at $x \in M$ equal to the Grassmannian ${\rm Gr}(r, T_x M)$ of $r$-dimensional subspaces in $T_x M$.
Recall that for a subspace $V \subset T_x M$ of dimension $r$ and the corresponding point $[V] \in {\rm Gr}(r, T_x M)$, there is a natural identification $$ T_{[V]} {\rm Gr}(r, T_x M) = \Hom (V, T_x M/V).$$
Define a holomorphic map $\gamma: Y \to {\rm Gr}(r, TM)$
by sending each $y \in Y$ to
 the $r$-dimensional subspace  ${\rm d}_y \phi (R_y) \subset T_{\phi(y)}M $.

 We claim that for $u \in T^{\phi}_y$, $${\rm d}_y \gamma (u) \in T_{\gamma(y)} {\rm Gr}(r, T_{\phi(y)}M) = \Hom ({\rm d}_y \phi (R_y), T_{\phi(y)}M /{\rm d}_y \phi (R_y))$$ is the homomorphism sending ${\rm d}_y \phi (v)$ to $-{\rm d}_y \phi ({\rm Levi}_y^R(u, v))$ for all $v \in R_y$.
 In fact, choose a local section $\tilde{u}$ of $T^{\phi}$ satisfying $\tilde{u}(y) = u$ and a local section $\tilde{v}$ of $R$ satisfying $\tilde{v}(y) = v$. Let $\{ \varphi_t, t \in \Delta\}$ be the  1-parameter group-germ of local biholomorphisms generated by the local vector field $\tilde{u}$ on $Y$. Since $\tilde{u}$ is tangent to the fibers of $\phi$, $${\rm d} \phi ({\rm d} \varphi_t (v)) - {\rm d} \phi (v) = 0,$$ for $ t\in \Delta$ sufficiently close to $0 \in \Delta$.
 Thus $$\frac{1}{t} {\rm d} \phi ({\rm d} \varphi_t (v) - \tilde{v}_{\varphi_t(y)}) = \frac{1}{t}({\rm d} \phi (v) - {\rm d} \phi (\tilde{v}_{\varphi_t(y)})).$$
 The right hand side  modulo ${\rm d}_y \phi (R_y)$ converges to $-  {\rm d}_y \gamma (u) ({\rm d}_y \phi (v))$ as $t$ approaches 0, while the left hand side converges to the image of the Lie derivative $$ {\rm d}_y \phi ({\rm Lie}_{\tilde{u}} \tilde{v}) = {\rm d}_y \phi ([\tilde{u}, \tilde{v}])$$ as $t$ approaches zero. This  modulo ${\rm d}_y \phi (R_y)$ is $ {\rm d}_y \phi ({\rm Levi}_y^R (u, v))$, which proves the claim.

  By the claim, the morphism $\gamma $ is constant on each fiber of $\phi$ under the assumption of the lemma and defines a section of ${\rm Gr}(r, TM) \to M$, i.e., a distribution $Q \subset TM$ of rank $r$. It is immediate that $R = \phi^{-1}Q$. \end{proof}

\begin{lemma}\label{l.reduction}
Let $L$ be a line bundle on a complex manifold $Y$ and let $Z \subset Y$ be a compact complex submanifold. Let $\vartheta: TY \to L$ be a nowhere vanishing $L$-valued 1-form on $Y$ which defines the distribution $ {\rm Ker}(\vartheta) \subset TY$ of corank 1. Assume that \begin{itemize} \item[(1)] ${\rm Null}({\rm Levi}_y^{{\rm Ker}(\vartheta)})\subset {\rm Ker}(\vartheta_y)$ has dimension $k$ for all $y \in Y$ for some integer $k \geq 0$ and
  \item[(2)] ${\rm Null}({\rm Levi}_y^{{\rm Ker}(\vartheta)}) \cap T_y Z =0$ for all $y \in Z$. \end{itemize}
Then there exist a neighborhood $O \subset Y$ of $Z$ and a
submersion $\nu:O \to X$ to a complex manifold $X$ with a contact structure $D \subset TX$ such that  $\dim X = \dim Y -k$,  ${\rm Ker}(\vartheta)|_{O}= \nu^{-1} D$ and $T^{\nu}_y = {\rm Null}({\rm Levi}_y^{{\rm Ker}(\vartheta)})$ for all $y\in O$. \end{lemma}

\begin{proof}
By the assumption (1), we have a distribution $\sN$ on $Y$ whose fiber at $y \in Y$ is ${\rm Null}({\rm Levi}_y^{{\rm Ker}(\vartheta)})$. It is easy to see that $\sN$ is integrable (in fact, this is the Cauchy characteristic of ${\rm Ker}(\vartheta)$ in the sense of Section 2 of \cite{Hw12}), defining a foliation of rank $k$ on $Y$ whose leaves are transversal to $Z$ by the assumption (2).
Pick a point $z \in Z$. We have a neighborhood $O_z \subset Y$ of $z$ with a submersion $\nu_z: O_z \to X_z$ whose fibers are leaves of $\sN$ on $O_z$. Since $Z$ is compact, there are finitely many points $z_1, \ldots, z_N \in Z$ such that $Z \subset \cup_{i=1}^N O_{z_i}$.
The fibers of two submersions $\nu_{z_i}$ and $\nu_{z_j}$ agree on the intersection $O_{z_i} \cap O_{z_j}$ for $1 \leq i, j \leq N$. Thus we can patch them to define a submersion $\nu: O \to X$ from a neighborhood $O \subset \cup_{i=1}^N O_{z_i}$ of $Z$ to a complex manifold $X$ of dimension $\dim Y -k$ with $T^{\nu} = \sN|_O$.   Applying  Lemma \ref{l.quotient} to $\nu$, we can find a contact structure $D$ on $X$ satisfying ${\rm Ker}(\vartheta)|_{O}= \nu^{-1} D$, proving the lemma. \end{proof}

\begin{proposition}\label{p.KZ} In the setting of Definition \ref{d.Levi}, let $\phi: D^{\perp} \to M$ be the natural projection and
 let ${\rm d} \theta^M$ be the natural symplectic form on $T^{\vee}M$ from Notation \ref{n.cotangent}. For  $x \in M$ and a nonzero $\alpha \in D^{\perp}_x$, define  $${\rm Null}({\rm d} \theta^M |_{T_{\alpha}D^{\perp}}) := \{ v \in T_{\alpha} D^{\perp}, \ {\rm d} \theta^M (v, u) =0 \mbox{ for all } u \in T_{\alpha} D^{\perp} \}.$$ Then \begin{itemize}
 \item[(i)] ${\rm Null}({\rm d} \theta^M |_{T_{\alpha}D^{\perp}})  \cap T^{\phi}_{\alpha} = 0$ and
     \item[(ii)]  the differential ${\rm d}_{\alpha} \phi: T_{\alpha} D^{\perp} \to T_x M$ gives an isomorphism   $$ {\rm Null}({\rm d} \theta^M|_{T_{\alpha} D^{\perp}}) \cong {\rm Null}(\alpha \circ {\rm Levi}_x^D) \subset D_x \subset T_x M.$$ In particular, we have the inclusion ${\rm Null}({\rm d} \theta^M|_{T_{\alpha} D^{\perp}})  \subset {\rm Ker}(\theta^M_{\alpha}).$ \end{itemize} \end{proposition}

\begin{proof}
This is essentially contained in Section 2.1 of \cite{Z99}. Since \cite{Z99} is presented in the setting of real differentiable manifolds with some notation different from ours, we give a full proof for the reader's convenience.

For each nonzero $\alpha \in D_x^{\perp}$, we define a subspace $$D^{\sharp}_{\alpha} \subset {\rm Ker}(\theta^M_{\alpha}) \subset T_{\alpha}(T^{\vee}M)$$ isomorphic to $D_x$ via the projection ${\rm d}_{\alpha} \pi_M: T_{\alpha}(T^{\vee}M) \to T_x M$ in the following way. For each $v \in D_x$, choose a local section $\vec{v}$ of $D$ in a neighborhood of $x$ with $v= \vec{v}(x)$. Let $\vec{h}_{\vec{v}}$ be the vector field on a neighborhood $O$ of $T_x^{\vee}M$ in $T^{\vee} M$ defined as in Proposition \ref{p.AS}.  It is easy to see that the value $\vec{h}_{\vec{v}} (\alpha)$ of $\vec{h}_{\vec{v}}$ at a point $\alpha \in D_x^{\perp}$ is determined by $v$, independent of the choice of $\vec{v}.$ Then the association $v \mapsto \vec{h}_{\vec{v}} (\alpha)$  defines a linear homomorphism $D_x \to T_{\alpha}(T^{\vee}M)$ whose image is the subspace $D^{\sharp}_{\alpha}.$

For each nonzero $\alpha \in D^{\perp}_x$, we claim
\begin{eqnarray}\label{e.Z99} {\rm Null} ({\rm d} \theta^M|_{T_{\alpha} D^{\perp}}) &=&  D^{\sharp}_{\alpha} \cap T_{\alpha} D^{\perp} \ \subset T_{\alpha}(T^{\vee} M) \end{eqnarray}
and \begin{equation}\label{e.KZ}  D^{\sharp}_{\alpha} \cap T_{\alpha} D^{\perp} \cong {\rm Null}(\alpha \circ {\rm Levi}^D_x) \subset D_x \end{equation} under the natural isomorphism $
D_x \cong D^{\sharp}_{\alpha}$. Certainly, (\ref{e.Z99}) and (\ref{e.KZ}) imply Proposition \ref{p.KZ}.

To prove the claim (\ref{e.Z99}), let $r$ be the rank of $D$ and let $\vec{v}_1, \ldots, \vec{v}_r$ be local sections of $D$ in a neighborhood $U \subset M$ of $x$  which gives a basis of the fibers of $D$ over $U$. They define functions $h_{\vec{v}_1}, \ldots, h_{\vec{v}_r}$ in a neighborhood of $\alpha \in T^{\vee}M$, as in Proposition \ref{p.AS}.   Then the germ of the submanifold $D^{\perp} \subset T^{\vee}M$ at $\alpha$ is the common zero set of the functions $h_{\vec{v}_1}, \ldots, h_{\vec{v}_r}$ in a neighborhood of $\alpha$. Note that a vector $v \in T_{\alpha} D^{\perp}$ is in ${\rm Null} ({\rm d} \theta^M|_{T_{\alpha}D^{\perp}})$ if and only if the 1-form ${\rm d}\theta^M(v, \cdot)$ on $T^{\vee} M$ annihilates $T_{\alpha} D^{\perp}$. This holds exactly when ${\rm d}\theta^M(v, \cdot)$ is in the linear span of ${\rm d} h_{\vec{v}_1}, \ldots, {\rm d} h_{\vec{v}_r}$ in $T^{\vee}_{\alpha}(T^{\vee} M)$. As ${\rm d} \theta^M$ is a symplectic form, this happens when $v$ is in the linear span of $\vec{h}_{\vec{v}_1}, \ldots, \vec{h}_{\vec{v}_r}$, i.e., when $v \in D^{\sharp}_{\alpha}.$ This proves (\ref{e.Z99}).

To prove the claim (\ref{e.KZ}),
let $\vec{v}$ be a local section of $D$ and let $\vec{h}_{\vec{v}}$ be the vector field on a neighborhood $O$ of $T^{\vee}_x M$ in $T^{\vee}M$ as defined in Proposition \ref{p.AS}.  By Proposition \ref{p.AS}, the vector $\vec{h}_{\vec{v}}(\alpha) \in D^{\sharp}_{\alpha}$ is tangent to $D^{\perp}$ if and only if the value of the 1-form ${\rm Lie}_{\vec{v}}\tilde{\alpha}$  at $x$ lies in $D^{\perp}_x$ for any local section $\tilde{\alpha} $ of $D^{\perp} \subset T^{\vee}M$ in a neighborhood of $x$ satisfying $\tilde{\alpha}_x = \alpha$.
This is equivalent to saying $${\rm Lie}_{\vec{v}} \tilde{\alpha} (\vec{w})_x = 0$$ for  any local section $\vec{w}$ of $D$ in a neighborhood of $x$. By Cartan formula, this is equivalent to $$ 0 = \vec{w} (\tilde{\alpha}(\vec{v}))_x + {\rm d} \tilde{\alpha} (\vec{v}, \vec{w})_x.$$
Since $\tilde{\alpha} (\vec{v}) \equiv 0 \equiv \tilde{\alpha}(\vec{w}),$ this is equivalent to $$ 0 = \tilde{\alpha}_x([ \vec{v}, \vec{w}]) =  \alpha \circ {\rm Levi}_x^D (\vec{v}, \vec{w}),$$ which says that $\vec{v}_x \in {\rm Null}(\alpha \circ {\rm Levi}_x^D)$. This proves (\ref{e.KZ}).   \end{proof}

We can reformulate Proposition \ref{p.KZ} in terms of contact forms as follows.

\begin{proposition}\label{p.null}
In Proposition \ref{p.KZ}, for each nonzero $\alpha \in D_x^{\perp}$, let $[\alpha] \in \BP D^{\perp} \subset \BP T^{\vee} M$ be the corresponding point in the projectivization of $T^{\vee} M$. Then for the $\sO_{\BP D^{\perp}}(1)$-valued 1-form $\vartheta^M|_{\BP D^{\perp}}$ on $\BP D^{\perp}$, we have an isomorphism
$$ {\rm Null} (\dd_{[\alpha]} (\vartheta^M|_{\BP D^{\perp}})) \cong {\rm Null}(\alpha \circ {\rm Levi}_x^D) \subset D_x$$ under the natural projection $\BP D^{\perp} \to M$. \end{proposition}

We use the following elementary lemma.

\begin{lemma}\label{l.codim1}
Let $\theta$ be a nowhere vanishing 1-form on a complex manifold $M$ such that   $${\rm Null}(({\rm d} \theta)_x) := \{ v \in T_x M, {\rm d} \theta (u, v) = 0 \mbox{ for all } u \in T_x M\}$$ is contained in ${\rm Ker}(\theta_x)$ for some $x \in M$.  Then ${\rm Null}(({\rm d} \theta)_x)$ is a subspace of ${\rm Null}(\dd_x \theta)$ of codimension at most 1.
\end{lemma}

\begin{proof}
It is immediate that ${\rm Null}(({\rm d}\theta)_x)$ is a subspace of ${\rm Null}(\dd_x \theta).$ Fix an element $w \in T_x M$ with $\theta(w) \neq 0$. Then
$${\rm Null}(({\rm d}\theta)_x) = \{ v \in {\rm Null}(\dd_x \theta), {\rm d} \theta (v, w) =0\}.$$ So it has codimension at most 1. \end{proof}

\begin{proof}[Proof of \ref{p.null}]
Put $D^+:= D^{\perp} \setminus (\mbox{0-section})$  and
let $\chi: D^+ \to \BP D^{\perp}$ be the submersion defining the projectivization. Then $${\rm Ker}(\theta^M|_{D^+}) = \chi^{-1} {\rm Ker}(\vartheta^M|_{\BP D^{\perp}})$$
and ${\rm Null}(\dd_{[\alpha]} (\vartheta^M|_{\BP D^{\perp}}))$ is the image of ${\rm Null}(\dd_{\alpha} (\theta^M|_{D^+}))$ under ${\rm d} \chi$ by Lemma \ref{l.inverse}.
Proposition \ref{p.KZ} (i) shows that ${\rm Null} ({\rm d} \theta^M|_{T_{\alpha} D^+}) $ is transversal to $T_{\alpha}^{\chi}$  and Proposition \ref{p.KZ} (ii)  implies  $${\rm Null}({\rm d} \theta^M|_{T_{\alpha} D^+}) + T_{\alpha}^{\chi} \  \subset \ {\rm Null}(\dd_{\alpha} (\theta^M|_{D^+})).$$ This inclusion must be an equality  by Lemma \ref{l.codim1}.
Thus ${\rm Null}(\dd_{[\alpha]} (\vartheta^M|_{\BP D^{\perp}}))$ is the image of ${\rm Null}({\rm d}\theta^M|_{T_{\alpha} D^+})$ which is isomorphic to ${\rm Null}(\alpha \circ {\rm Levi}_x^D) \subset D_x$ by Proposition \ref{p.KZ}. \end{proof}

\section{Contact lines on the Heisenberg group}\label{s.Heis}

\begin{definition}\label{d.Heisenberg}
Let $\I$ be a complex vector space of dimension 1 and let $\W$ be a complex vector space of dimension $2m \geq 4$.
Let $\omega: \wedge^2 \W \to \I$ be a symplectic form  on $\W$.
The {\em Heisenberg group} $\BH$ is the algebraic group whose underlying variety is the affine variety of the vector space $\W \oplus \I$ and whose group multiplication is given by $$(w,z) \cdot (w', z') = (w+w', z+z' + \frac{1}{2} \omega(w,w'))$$ for all $w,w' \in \W$ and $z, z' \in \I$.
The Lie algebra of the Heisenberg group is the {\em Heisenberg algebra}  whose underlying vector space is the direct sum $\W \oplus \I$ and whose Lie bracket is  given by $$[(w, z), (w', z')] = \omega (w, w')$$ for all $w,w' \in \W$ and $z, z' \in \I$.
Let $o =(0,0) \in \BH$ be the identity element of $\BH$. The vector space $\W \subset T_o \BH$ determines a left-invariant subbundle $\sW \subset T \BH$, which is a contact structure on $\BH.$ The left $\BH$-action gives a trivialization of the line bundle $T\BH/\sW \cong \I \times \BH$ such that $\sW = {\rm Ker}(\btheta)$ for a left-invariant 1-form ${\btheta}$ on $\BH$.  \end{definition}

\begin{definition}\label{d.line}
For a nonzero vector $w \in \W$, define the curve $\ell^w_o \subset \BH$ as the underlying variety of the 1-parameter subgroup $\exp(\C w) \subset \W \subset \BH$: $$ \ell^w_o := \{ (tw, 0) \in \BH, t \in \C\}.$$ For any point $x = (u,z) \in \BH$, define $\ell^w_x$ as the locus
$$x \cdot \exp(\C w) = \{(u + tw, z + \frac{t}{2} \omega(u,w), \ t \in \C\}.$$ A $w$-{\em line} on $\BH$ means $\ell^w_x$ for some  $x \in\BH$.
An affine line on $\W \oplus \I$ is called a {\em contact line} if it is a $w$-line for some nonzero $w \in \W.$
Tangent spaces of contact lines determine a line subbundle $\sF \subset T \BP \sW$, i.e., a foliation of rank 1 on $\BP \sW$.
\end{definition}

\begin{definition}\label{d.sL}
For each $[w] \in \BP \W$, let $A_{[w]} \subset \BH$ be the algebraic subgroup $\exp (\C w)$ and let $A \subset \BP \W \times \BH$ be the disjoint union of such subgroups parametrized by $\BP \W$, namely, $$ A := \{([w], g) \in \BP \W \times \BH, \ g \in A_{[w]}\}.$$
We have the projection $\eta: A \to \BP \W$ whose fiber $\eta^{-1}([w]) \subset A$ at  $[w] \in \BP \W$ is isomorphic to $A_{[w]} \subset \BH$ via the projection $A \to \BH$ .  Let $\lambda: \sL \to \BP \W$ be the fiber bundle over $\BP \W$ whose  fiber $\lambda^{-1}([w])$ at $[w] \in \BP \W$ is the coset space $\BH/A_{[w]}$.  In other words, we have a sequence of morphisms
$$ \begin{array}{ccccc}  A & \subset & \BP \W \times \BH &  \to & \sL \\ \downarrow
& & \downarrow & & \downarrow \lambda \\ \BP \W & =  & \BP \W & = & \BP \W\end{array} $$ which gives over each point $[w] \in \BP \W$  the quotient by the subgroup $$ A_{[w]} \subset \BH \to \BH/A_{[w]}.$$
By the left $\BH$-action, there is a natural trivialization of the projective bundle on $\BH$ \begin{equation}\label{e.PW}  \BP \sW \ \stackrel{\cong}{\longrightarrow} \ \BP \W \times \BH. \end{equation} The pair of morphisms $(\varrho, \upsilon)$ given by the projection $\varrho: \BP \sW \to \sL$ induced by (\ref{e.PW}) and the natural projection $ \upsilon: \BP \sW \to \BH$ can be viewed as the universal family of contact lines on $\BH$. In other words, we can regard  $\sL$ as the set of all contact lines on $\BH$.  The fiber $\lambda^{-1}([w])$ of $\lambda: \sL \to \BP \W$ corresponds to the set of all $w$-lines  on $\BH$.
\end{definition}

\begin{definition}\label{d.Legendrian}
For a complex submanifold $S \subset \BP \W$ (not necessarily closed), denote by $S^+$ the set of nonzero vectors in $\W$ corresponding to points of $S$, i.e., the affine cone over $S$ with the zero point removed. \begin{itemize}
\item[(1)] We say that $S$ is  a {\em Legendrian submanifold} if $\dim S = m-1$ and for each $\alpha \in S^+$, the affine tangent space $T_{\alpha} S^+ \subset \W$ satisfies $$ \omega(\alpha, T_{\alpha} S^+) = 0.$$ \item[(2)] Define for $\alpha \in \W$,  $$\alpha^{\perp \omega} := \{ w \in \W, \omega (\alpha, w) =0\}. $$ Then for a  Legendrian submanifold $S \subset \BP \W$ and a point $\alpha \in S^+$, we have $ T_{\alpha} S^+ \subset \alpha^{\perp \omega}.$
\item[(3)] An $S$-{\em line} on $\BH$ means a $w$-line for some $w \in S^+$.
The tangent spaces of $S$-lines form a fiber subbundle $\sS \subset \BP \sW$ with the natural projection $\mu: \sS \to \BH$, which corresponds to $S \times \BH$ under (\ref{e.PW}) in Definition \ref{d.sL}. \item[(4)]
In (3), define $\sL^S \subset \sL$ as the image $\varrho(\sS)$ and $\rho: \sS \to \sL^S$ as the restriction to $\sS$ of $\varrho: \BP \sW \to \sL$ in Definition \ref{d.sL}. Then the pair of morphisms  $(\rho, \mu)$ can be viewed as the universal family of $S$-lines on $\BH$. \item[(5)] In (3),
denote by $\sS^+ \subset \sW \setminus (\mbox{ 0-section })$, the left-invariant fiber bundle with each fiber isomorphic to $S^+$.  We have  the natural projection $\chi: \sS^+ \to \sS$ which is a $\C^{\times}$-bundle.
 \end{itemize} \end{definition}

\begin{definition}\label{d.R}
Using the terminology of Definition \ref{d.Legendrian}, assume that $S \subset \BP \W$ is a  Legendrian submanifold.
Define a distribution $R \subset \mu^{-1} \sW \subset T\sS$ on $\sS$ as follows. At a point $[\alpha] \in \sS$ with $\alpha \in \sS^+ \subset \sW$ and $\mu([\alpha]) = x \in \BH$, the fiber $R_{[\alpha]}$ is
$$R_{[\alpha]}:= \{ v \in (\mu^{-1} \sW)_{[\alpha]} \subset T_{[\alpha]}\sS, \ {\rm d} {\btheta} ({\rm d} \mu (v), \alpha) = 0\}$$ where ${\btheta}$ is the left-invariant 1-form on $\BH$ in Definition \ref{d.Heisenberg} and $\alpha$ is viewed as a nonzero vector in $\sW_x \subset T_x \BH$. Then $R$ is a vector subbundle in $T\sS$ of corank 2. \end{definition}

We state two results on the distribution $R$ on $\sS$:

\begin{proposition}\label{p.Q}
In Definition \ref{d.R}, there exists a vector subbundle $Q \subset T \sL^S$ of corank 2 such that $R = \rho^{-1} Q$. \end{proposition}

\begin{proposition}\label{p.R}
In Definition \ref{d.R},
 the subspace ${\rm Null}(\alpha \circ {\rm Levi}_y^R) \subset R_y$ has dimension $m$ for  any $y \in \sS$ and any nonzero vector $\alpha \in R_y^{\perp}$. \end{proposition}

Postponing the proofs of Propositions \ref{p.Q} and \ref{p.R} to Section \ref{s.coordi}, we derive the following consequence.

\begin{corollary}\label{c.N}
For the distribution $Q$  on $\sL^S$ from Proposition \ref{p.Q},
 $$ \dim {\rm Null}(\dd_{[\beta]}(\vartheta^{\sL^S}|_{\BP Q^{\perp}})) = m-1$$ for any nonzero $\beta \in  Q^{\perp}$.
\end{corollary}

\begin{proof}
From $R = \rho^{-1}Q$ in Proposition \ref{p.Q}, the differential of  the submersion $\rho: \sS \to \sL^S$ induces an  isomorphism of projective bundles on $\sS$ $$\rho^*\BP Q^{\perp} \stackrel{\cong}{\longrightarrow} \BP R^{\perp}.$$ This induces  a submersion  $\eta: \BP R^{\perp} \to \BP Q^{\perp}$ of relative dimension 1 satisfying the commutative diagram
$$ \begin{array}{ccccccc}
\BP T^{\vee} \sS & \supset & \BP R^{\perp} & \stackrel{\eta}{\to} & \BP Q^{\perp} & \subset & \BP T^{\vee} \sL^S \\ \varpi_{\sS} \downarrow & &
\downarrow & & \downarrow & & \downarrow \varpi_{\sL^S} \\
\sS & = & \sS & \stackrel{\rho}{\to} & \sL^S & = & \sL^S. \end{array} $$
As $\eta$ comes from the derivative of $\rho$, we have \begin{equation}\label{e.eta} \eta^* (\vartheta^{\sL^S}|_{\BP Q^{\perp}}) = \vartheta^{\sS}|_{\BP R^{\perp}}.\end{equation}
By Lemma \ref{l.pullback} and (\ref{e.eta}),  it suffices to prove that
$$\dim {\rm Null} (\dd_{[\alpha]} (\vartheta^{\sS}|_{\BP R^{\perp}})) = m$$ for all $[\alpha] \in \BP R^{\perp}$. By Proposition \ref{p.null}, this follows from $\dim {\rm Null}(\alpha \circ {\rm Levi}_y^R) = m$ in Proposition \ref{p.R}. \end{proof}

The following is a more precise version of Theorem \ref{t.main}.

\begin{theorem}\label{t.precise}
In the setting of Propositions \ref{p.Q} and \ref{p.R}, choose a point $w \in S^+$ and a $w$-line $\ell$ on $\BH$. Then there exist
\begin{itemize}
\item[(1)] a complex manifold $X$ of dimension $2m+1$ with a contact structure $D \subset TX$;
\item[(2)] an unbendable rational curve $C \subset X$ subordinate to $D$;
\item[(3)] an (Euclidean) open neighborhood $U \subset \BH$ of $\ell$; and
\item[(4)]  an open embedding $U \subset X$
\end{itemize}
such that $ C \cap U = \ell, D|_U = \sW|_U$ and the VMRT $\sC^C_x$ for $x \in U \cap C$ is projectively isomorphic to the germ $[w] \in S$ corresponding to $w \in S^+$. \end{theorem}

For the proof, we recall the following characterization of unbendable rational curves.

\begin{lemma}\label{l.char}
Let $C \subset X$ be a nonsingular rational curve on a complex manifold of dimension $n$ whose normal bundle $N_{C/X}$ is isomorphic to $$\sO(a_1) \oplus \cdots \oplus \sO(a_{n-1})$$
for some nonnegative integers $a_1 \geq \cdots \geq a_{n-1}.$
Set $p := \sum_{i=1}^{n-1} a_i$. Suppose there is a $p$-dimensional family of deformations of $C$ fixing a point $x \in C$ such that the tangent directions to these deformations form a $p$-dimensional germ of a submanifold in $\BP T_x X$. Then  $$ a_1 = \cdots = a_{p} =1, \ a_{p+1} = \cdots = a_{n-1} =0.$$  \end{lemma}

\begin{proof}
 The space $\sK_x$ of all small deformations of $C$ in $X$ fixing $x$ is nonsingular of dimension $p$ because
$$H^0(C, N_{C/X} \otimes {\bf m}_x) = p \mbox{ and }
H^1(C, N_{C/X} \otimes {\bf m}_x) =0.$$
Let $\tau_x : \sK_x \to \BP T_x X$ be the morphism sending a curve to its tangent direction at $x$.  From our assumption,  the image of $\tau_x$ is nonsingular of dimension $p$, hence $\tau_x$ is an immersion.
It is easy to see (e.g. see the argument in the proof of Proposition 1.4 in \cite{Hw01}) that $\tau_x$ is an immersion if and only if
$$H^0(C, N_{C/X} \otimes {\bf m}_x^2) =0.$$ This implies  $ a_1 = \cdots = a_{p} =1, \ a_{p+1} = \cdots = a_{n-1} =0.$
 \end{proof}

\begin{proof}[Proof of Theorem \ref{t.precise}]
The left-action of $\BH$ on itself induces $\BH$-actions on both $\sS$ and $\sL^S$ such that the morphisms $\rho: \sS \to \sL^S$ and $\mu: \sS \to \BH$ are equivariant under the $\BH$-actions. The distributions $R$ and $Q$ are also $\BH$-equivariant and so is the commutative diagram
 $$\begin{array}{ccccccc}
\BP T^{\vee} \sS & \supset & \BP R^{\perp} & \stackrel{\eta}{\to} & \BP Q^{\perp} & \subset & \BP T^{\vee} \sL^S \\ \varpi_{\sS} \downarrow & &
\downarrow & & \downarrow \psi & & \downarrow \varpi_{\sL^S} \\
\sS & = & \sS & \stackrel{\rho}{\to} & \sL^S & = & \sL^S \end{array} $$
in the proof of Corollary \ref{c.N}.

Let $\btheta$ be the left-invariant 1-form on $\BH$ defining the distribution $\sW \subset T\BH$.  Let $\sigma: \sS \to \BP R^{\perp}$ be the section of the projection $\BP R^{\perp} \to \sS$ determined by $\mu^*{\btheta}$.
Define the morphism $\iota: \sS \to \BP Q^{\perp}$ by the composition $\eta \circ \sigma$. We have the $\BH$-equivariant commutative diagram
$$ \begin{array}{ccc}  \sS & \stackrel{\iota}{\to} & \BP Q^{\perp} \\
\rho \downarrow & & \downarrow \psi \\
\sL^S & = & \sL^S. \end{array} $$
  Lemma \ref{l.sigma} and (\ref{e.eta}) imply
\begin{equation}\label{e.theta} \iota^* (\vartheta^{\sL^S}|_{\BP Q^{\perp}}) = \mu^* {\btheta}. \end{equation}
Since ${\btheta}$ is a contact form on $\BH$ and $\dim T^{\mu}_{[\alpha]} = m-1$ for any $[\alpha] \in \sS,$ Lemma \ref{l.pullback} gives
$$\dim {\rm Null} (\dd_{[\alpha]} (\mu^* {\btheta})) = m-1.$$
As $ \dim {\rm Null}(\dd_{[\beta]}(\vartheta^{\sL^S}|_{\BP Q^{\perp}})) = m-1$ from Corollary \ref{c.N}, Lemma \ref{l.pullback} and (\ref{e.theta}) imply that $\iota$ must be an open immersion at every point of $\sS$.
   Considering the action  of the subgroup $A_{[v]} \subset \BH$ for $[v] \in S \subset \BP \W$ in Definition \ref{d.sL} and using the  $\BH$-equivariance, we see that the morphism $\iota$ embeds a fiber $\rho^{-1}(z)$ for any $z \in \sL^S$ with $\lambda (z) = [v] \in S$  into the projective line $\psi^{-1}(z)$ as an affine open set.   This implies that $\iota$ is an open embedding such that the complement $\BP Q^{\perp} \setminus \iota(\sS)$ is a section of the projection $\psi: \BP Q^{\perp} \to \sL^S$.

From Lemma \ref{l.reduction} and Corollary \ref{c.N}, we can choose a neighborhood $O$ of the fiber $C'$ of $\psi$ at the point $[\ell] \in \sL^S$ with a submersion $\nu: O \to X$ to a complex manifold $X$ with a contact structure $D \subset TX$ such that
$$T^{\nu}_{[\beta]} = {\rm Null}(\dd_{[\beta]}(\vartheta^{\sL^S}|_{\BP Q^{\perp}})) $$ for each $[\beta] \in O \subset \BP Q^{\perp}$ and
$$\nu^{-1} D = {\rm Ker}(\vartheta^{\sL^S}|_{\BP Q^{\perp}}).$$
Furthermore,  (\ref{e.theta}) implies that the images of a fiber of $\mu$ under $\iota$ intersects $O$ along a fiber of $\nu$, inducing an open embedding of a neighborhood $U$ of $\ell$ into $X$ such that  $D|_U = \sW|_U$.
Let $C \subset X$ be the image  $\nu (C')$.
Then $C$ is a nonsingular rational curve in $X$ and
the normal bundle of $C \subset X$ is isomorphic to $$\sO(a_1) \oplus \cdots \oplus \sO(a_{2m})$$
for some nonnegative integers $a_1 \geq \cdots \geq a_{2m}$
 because the normal bundle of $C'$ in  $O$ is trivial.
The degree of $C$ with respect to the line bundle $L:= TX/D$ is 1 because the line bundle $T \BP Q^{\perp}/{\rm Ker}(\vartheta^{\sL^S}|_{\BP Q^{\perp}})$ is isomorphic to  $\sO_{\BP Q^{\perp}}(1)$. As $D$ is a contact structure, we have $K_X^{-1} \cong L^{\otimes (m+1)}$ (e.g. by (2.2) of \cite{LB})
implying $\sum_{i=1}^{2m} a_i = m-1$. Thus the space of all small deformations of $C$ in $X$ has dimension equal to $$\dim H^0(\BP^1, \sO(a_1) \oplus \cdots \oplus \sO(a_{2m})) = \sum_{i=1}^{2m} a_i + 2m = 3m-1.$$ Since  the two morphisms $\sL^S \leftarrow \sS \to \BH$ describe the family of contact lines on $\BH$ effectively with $\dim \sL^S = 3m-1$ and are compatible with $\sL^S \leftarrow O \to X$ via $\iota$, we see that
compact fibers of $O \to \sL^S$ describe the $(3m-1)$-dimensional family of small deformations of the rational curve $C$ in $X.$
Moreover, the intersection of $U$ and the small deformations of $C$ fixing a point $x\in U \cap C$ correspond to an open subset in the family of $S$-lines through $x$. Thus their tangent directions at $x$ form a submanifold in $\BP T_x X$ projectively isomorphic to the germ of $S$ at $[w]$, where the  tangent direction $[T_x \ell] \in \sS$ of the affine line $\ell = U \cap C$ at $x$ is identified with $([w], x) \in \BP \W \times \BH$ via the inclusion $\sS \subset \BP \sW$ and the isomorphism (\ref{e.PW}). By Lemma \ref{l.char}, we see that $$a_1 = \cdots = a_{m-1} =1, \ a_{m} = \cdots = a_{2m} =0$$ and $C$ is unbendable with the VMRT $\sC^C_x$ projectively isomorphic to the germ $[w] \in S$.
\end{proof}

\section{Computation in coordinates on the Heisenberg group}\label{s.coordi}

In this section, we prove Propositions \ref{p.Q} and \ref{p.R} by computation in coordinates. Let us choose coordinates on $\BH$ as follows.

\begin{notation}\label{n.coordi}
 In Definition \ref{d.Heisenberg}, choose a  basis $\bw_1, \ldots, \bw_{2m}$ of $\W$ and a nonzero vector $\bi \in \I$ satisfying
$$\omega(\bw_j, \bw_k) = \omega(\bw_{m+j}, \bw_{m+k}) = 0 \mbox{ and }
\omega(\bw_j, \bw_{m+k}) = \delta_{j,k} \bi$$ for all $ 1 \leq j, k \leq m .$
Let $(x_1, \ldots, x_{2m}, x_{2m+1})$ be the dual coordinates of the basis $\{ \bw_1, \ldots, \bw_{2m}, \bi\}$ of the vector space $\W \oplus \I$ and view them as coordinates on the affine variety underlying $\BH$. \end{notation}

The following is straightforward.

\begin{lemma}\label{l.coordi} In terms of the coordinates in Notation \ref{n.coordi},
the multiplication $(x_1, \ldots, x_{2m+1}) \cdot (c_1, \ldots, c_{2m+1})$ of the two elements    $$ (x_1, \ldots, x_{2m+1}) \mbox{ and } (c_1, \ldots, c_{2m+1}) \in \BH$$ is given by   $$(x_1+ c_1, \ldots, x_{2m} + c_{2m}, x_{2m+1} + c_{2m+1} + \frac{1}{2} \sum_{j=1}^m(c_{m+j} x_{j} - c_{j} x_{m+j}))$$ and
a  left-invariant 1-form ${\btheta}$ on $\BH$ can be chosen as $${\btheta} = \sum_{j=1}^m (x_j {\rm d} x_{m+j} - x_{m+j} {\rm d} x_j) \ - 2 \ {\rm d} x_{2m+1}.$$
    \end{lemma}

\begin{proposition}\label{p.vecf}
In terms of the coordinates $x_1, \ldots, x_{2m+1}$ of Notation \ref{n.coordi}, the 1-forms ${\rm d} x_1, \ldots, {\rm d} x_{2m}$ determine regular functions on $\sW \subset T\BH$ denoted by $\lambda_1, \ldots, \lambda_{2m}$ such that $$(x_1, \ldots, x_{2m+1}, \lambda_1, \ldots, \lambda_{2m})$$ give affine coordinates on the nonsingular variety $\sW$.
\begin{itemize} \item[(i)]
Define the vector field on $\sW$
$$F := \sum_{j=1}^m (\lambda_j \frac{\p}{\p x_j} + \lambda_{m+j} \frac{\p}{\p x_{m+j}}) + \frac{1}{2} \sum_{j=1}^m (\lambda_{m+j} x_j - \lambda_j x_{m+j}) \frac{\p}{\p x_{2m+1}}.$$
Then  the natural projection $\sW \setminus (0-\mbox{section}) \to \BP \sW$ sends the line subbundle spanned by $F$ in $T\sW$ to $\sF \subset T
\BP \sW$ in Definition \ref{d.line}. In particular, the fibers of $\varrho$ in Definition \ref{d.sL} correspond to  the leaves of the vector field $F$.
\item[(ii)] The 1-form
$$\sum_{j=1}^m (\lambda_j {\rm d} \lambda_{m+j} - \lambda_{m+j} {\rm d} \lambda_j)$$ on $\sW$ annihilates the tangent vectors of the submanifold $\sS^+ \subset \sW$ of Definition \ref{d.Legendrian}.
\item[(iii)]  Define a 1-form $\zeta$ on $\sW$ by $$\zeta:=  \sum_{j=1}^m (\lambda_j  {\rm d} x_{m+j} - \lambda_{m+j} {\rm d} x_j)$$  and
a 1-form $\theta$ on $\sS^+$ by       $$ \theta : =  (\mu \circ \chi)^* \btheta,$$  where $\chi: \sS^+ \to \sS$ and $\mu: \sS \to \BH$ are  as in Definition \ref{d.Legendrian}.  Then the distribution $\chi^{-1}R$ of corank 2 on $\sS^+$ is defined by the two point-wise independent 1-forms
    $\theta$ and $\zeta|_{\sS^+}$.
\end{itemize}
\end{proposition}

\begin{proof}
For any point $w= (c_1, \ldots, c_{2m}) \in \W$,
the $w$-line through a point $(x_1, \ldots, x_{2m+1}) \in \BH$  is given by
$$(x_1+ t c_1, \ldots, x_{2m}+ t c_{2m}, x_{2m+1} + \frac{t}{2} \sum_{j=1}^n (c_{m+j} x_j - c_j x_{m+j}))$$ with $t \in \C$
by Lemma \ref{l.coordi}. The tangent vector field $\frac{{\rm d}}{{\rm d} t}$ along this $w$-line describe a curve in $\sW$ with $\lambda_1 =c_1, \ldots, \lambda_{2m} = c_{2m}$ taking constant values. It is straightforward to see that the vector field $F$ is tangent to this curve in $\sW$.  It follows that leaves of $F$ give leaves of $\sF$, proving (i).

The functions $\lambda_1, \ldots, \lambda_{2m}$ are invariant under the left $\BH$-action on $\sW$. Thus the germ of $\sS^+$ as a local subvariety in $\sW$ is defined by local analytic functions in $\lambda_1, \ldots, \lambda_{2m}$, which do not depend on the coordinates $x_1, \ldots, x_{2m+1}.$ Thus to prove (ii), it suffices to show that the expression  $$\sum_{j=1}^m (\lambda_j {\rm d} \lambda_{m+j} - \lambda_{m+j} {\rm d} \lambda_j)$$ regarded as a 1-form on $\W = \sW_o$ annihilates the tangent vectors of  $S^+ \subset \W = \sW_o.$ Since $$\omega = \sum_{j=1}^m {\rm d} \lambda_j \wedge {\rm d} \lambda_{m+j}$$ on $\sW_o = \W$, this follows from Definition \ref{d.Legendrian} (1).

To check (iii), note that the distribution $(\mu \circ \chi)^{-1} \sW$ is defined by $\theta$ and the subbundle $\chi^{-1}R \subset (\mu \circ \chi)^{-1} \sW$  is defined by one additional condition ${\rm d} \theta (F, \cdot) =0.$ The last equation is reduced to $\zeta =0$ from the expression of $\btheta$ in Lemma \ref{l.coordi} and the expression of $F$ in (i).
\end{proof}

Now we are ready to give proofs of Propositions \ref{p.Q} and \ref{p.R}.

\begin{proof}[Proof of \ref{p.Q}]
As the vector field $F$ is tangent to $\sS^+$, to prove Proposition \ref{p.Q}, it suffices to show that $[F, \chi^{-1}R] \subset \chi^{-1}R$ by Lemma \ref{l.quotient}. The Lie derivatives of the 1-forms on $\sW$ in Proposition \ref{p.vecf} (iii) by the vector field $F$ in Proposition \ref{p.vecf}(i) are  $${\rm Lie}_F  \theta = {\rm d} \theta (F, \cdot) = 2 \zeta$$ and $${\rm Lie}_F \zeta = {\rm d} \zeta (F, \cdot) = \sum_{j=1}^m (\lambda_j {\rm d} \lambda_{m+j} - \lambda_{m+j} {\rm d} \lambda_j),$$  because $\theta(F) \equiv 0 \equiv \zeta(F)$.  The last expression is zero on $\sS^+$ by Proposition \ref{p.vecf} (ii).  As $\chi^{-1} R$ is defined by $\theta = \zeta =0$ on $\sS^+$, this shows that $[F, \chi^{-1} R] \subset \chi^{-1} R.$ \end{proof}

\begin{proof}[Proof of Proposition \ref{p.R}]
Pick a point $z \in \sS^+ \subset \sW$ with $\chi(z) = y$.
As $\chi: \sS^+ \to \sS$ has 1-dimensional fibers, it suffices by Lemma \ref{l.pullback} to show that for any nonzero $\beta \in (\chi^{-1}R)^{\perp}$,
\begin{equation}\label{e.perp}
\dim {\rm Null}(\beta \circ {\rm Levi}_z^{\chi^{-1}R}) = m+1. \end{equation}

By the $\BH$-action, we may assume that $\mu(y) = o$.  Using the fact that the symplectic group of $(\W, \omega)$ acts transitively on the set of Lagrangian subspaces of $\W,$ we can choose the affine coordinates on $\sW$ $$(x_1, \ldots, x_{2m+1}, \lambda_1, \ldots, \lambda_{2m})$$
in Proposition \ref{p.vecf} such that  $z$ is given by $$x_1 = \cdots = x_{2m+1} =  \lambda_1 = \cdots = \lambda_{m-1} = \lambda_{m+1} = \cdots = \lambda_{2m} =0, \ \lambda_m =1$$ and  the tangent space of $\sS^+  $ at $z$  is given by $$\lambda_{m+1} = \cdots = \lambda_{2m} =0.$$

Note that for any nonzero $\beta \in (\chi^{-1}R)_z^{\perp}$, we have \begin{equation}\label{e.beta} \beta \circ {\rm Levi}_z^{\chi^{-1}R} = {\rm d} \tilde{\beta} |_{\wedge^2 (\chi^{-1} R)_z}, \end{equation} where $\tilde{\beta}$ is a local 1-form on $\sS^+$ with $\tilde{\beta}(z) = \beta$.
From Proposition \ref{p.vecf} (iii), we have \begin{equation}\label{e.langle} (\chi^{-1}R)_z = \langle \frac{\p}{\p x_1}, \ldots, \frac{\p}{\p x_{2m-1}}, \frac{\p}{\p \lambda_1}, \ldots, \frac{\p}{\p \lambda_m} \rangle. \end{equation}

Applying (\ref{e.beta}) to $\beta = \theta_z \in (\chi^{-1} R)^{\perp}_z$ for $\theta$ in Proposition \ref{p.vecf} (iii), we have  \begin{eqnarray*} {\rm Null} (\theta_z \circ {\rm Levi}_z^{\chi^{-1} R}) & = & {\rm Null}(\sum_{j=1}^m {\rm d} x_j \wedge {\rm d} x_{m+j}|_{\wedge^2 (\chi^{-1} R)_z}) \\
& = & {\rm Null}(\sum_{j=1}^{m-1} {\rm d} x_j \wedge {\rm d} x_{m+j}|_{\wedge^2 (\chi^{-1} R)_z}). \end{eqnarray*} It follows that $${\rm Null}(\theta_z \circ {\rm Levi}_z^{\chi^{-1}R}) = \langle \frac{\p}{\p x_m}, \frac{\p}{\p \lambda_1}, \ldots, \frac{\p}{\p \lambda_{m}} \rangle,$$ which has dimension $m+1$. This verifies (\ref{e.perp}) for $\beta = \theta_z$.

Next, applying (\ref{e.beta}) to $\beta = (\zeta + c \theta)_z \in (\chi^{-1} R)_z^{\perp}$ with any $c \in \C$, we have $${\rm Null}((\zeta + c \theta)_z \circ {\rm Levi}_z^{\chi^{-1}R}) = {\rm Null}({\rm d} \zeta + c \ {\rm d} \theta)|_{\wedge^2 (\chi^{-1}R)_z}.$$ Since
${\rm d} \zeta + c \ {\rm d} \theta$ is  $$2 c \sum_{j=1}^m {\rm d} x_j \wedge {\rm d} x_{m+j} + \sum_{j=1}^m ({\rm d} \lambda_j \wedge {\rm d} x_{m+j} - {\rm d} \lambda_{m+j} \wedge {\rm d} x_j),$$ its restriction to (\ref{e.langle}) is
 $$2 c \sum_{j=1}^{m-1} {\rm d} x_j \wedge {\rm d} x_{m+j} + \sum_{j=1}^{m-1} {\rm d} \lambda_j \wedge {\rm d} x_{m+j}.$$ Its
 null space on (\ref{e.langle})
  has a basis given by $$\frac{\p}{\p x_m}, \frac{\p}{\p \lambda_m} \mbox{ and } \frac{\p}{\p x_{j}} - 2 c \frac{\p}{\p \lambda_j} \mbox{ for } 1 \leq j \leq m-1.$$ Thus it has dimension $m+1$, verifying (\ref{e.perp}) for $\beta= (\zeta + c \theta)_z$ for any $c \in \C$. This completes the proof of (\ref{e.perp}).
\end{proof}

\bigskip
{\bf Acknowledgment}
I would like to thank Qifeng Li for valuable comments on the first draft of the paper and the referee for helpful suggestions to improve the presentation.

\bigskip
Institute for Basic Science

Center for Complex Geometry

Daejeon, 34126

Republic of Korea

jmhwang@ibs.re.kr
\end{document}